\documentclass{article}
\usepackage{amsfonts}
\usepackage{amssymb}
\usepackage{amsthm}								 
\usepackage{latexsym}
\usepackage{ae}                    
\usepackage{tabularx}              
\usepackage{enumerate}             
\usepackage{enumitem}							 
\usepackage{bbm}									 
\usepackage[arrow,matrix,curve]{xy}
\usepackage{makeidx}
\usepackage[final,dvips]{graphicx}
\usepackage{fancyhdr}
\usepackage[latin1]{inputenc}
\usepackage{amsmath}

\newtheoremstyle{SatzmitAbstand}
{10pt}    
{10pt}    
{\itshape}     
{25pt}    
{\bfseries} 
{:}     
{5pt} 
{}     
 
\theoremstyle{SatzmitAbstand}

\newtheorem{theorem}{Theorem}[section] 
\newtheorem{proposition}[theorem]{Proposition}
\newtheorem{corollary}[theorem]{Corollary}

\newtheorem{lemma}[theorem]{Lemma}
\newtheorem{remark}[theorem]{Remark}
\newtheorem{definition}[theorem]{Definition}

\newtheorem*{theorem*}{Theorem}

\newcommand{\KK}{\mathcal{K}}

\newcommand{\CC}{\mathbb{C}}
\newcommand{\RR}{\mathbb{R}}
\newcommand{\NN}{\mathbb{N}}

\begin{document}


\begin{center}
{\bf \Large The Blackadar-Handelman theorem for non-unital C*-algebras}
\end{center}

\vspace{0.3cm}

\begin{center}
{\it \large Henning Petzka}
\end{center}

\vspace{0.3cm}

\begin{abstract}
A well-known theorem of Blackadar and Handelman states that every unital stably finite C*-algebra has a bounded quasitrace. Rather strong generalizations of stable finiteness to the non-unital case can be obtained by either requiring the multiplier algebra to be stably finite, or alternatively requiring it to be at least stably not properly infinite. This paper deals with the question whether the Blackadar-Handelman result can be extended to the non-unital case with respect to these generalizations of stable finiteness.

For suitably well-behaved C*-algebras there is a positive result, but none of the non-unital versions holds in full generality. Two examples of C*-algebras are constructed. The first one is a non-unital, stably commutative C*-algebra $A$ that contradicts the weakest possible generalization of the Blackadar-Handelman theorem: The multiplier algebras of all matrix algebras over $A$ are finite, while $A$ has no bounded quasitrace.
The second example is a non-unital, simple C*-algebra $B$ that is stably non-stable, i.e. no matrix algebra over $B$ is a stable C*-algebra. In fact, the multiplier algebras over all matrix algebras of this C*-algebra are not properly infinite. Moreover, the C*-algebra $B$ has no bounded quasitrace and therefore gives a simple counterexample to a possible generalization of the Blackadar-Handelman theorem.
\end{abstract}

\section{Introduction}

In the classification theory of separable, exact C*-algebras there are a number of regularity properties of C*-algebras, the most prominent of which is absorption of the Jiang-Su algebra. Two strictly weaker properties are the corona factorization property, and a property introduced by R\o rdam in \cite{R6}, that is usually simply called `regularity'. The corona factorization property may be defined in several equivalent ways, one of which is to require each projection that is full in the multiplier algebra of the stabilization of the algebra in question to be properly infinite. R\o rdam's property requires that every full hereditary subalgebra with no non-zero bounded trace and no non-zero unital quotient must be stable. It is well known (see e.g. \cite{Ng}) that the latter property is at least as strong as the corona factorization property. Little is known about the converse (see for example \cite{Ng}), and this paper arose out of a study of that question.\\
 
\noindent In a survey (\cite{Ng}) on the corona factorization property, Ng defines another regularity property, which requires each full hereditary subalgebra with no non-zero bounded trace and no non-zero unital quotient to be only stably stable (i.e., it requires all sufficiently large matrix algebras over such a subalgebra to be stable). (If one matrix algebra is stable, then all larger matrix algebras are stable (\cite {R2} Proposition 2.1).) Ng shows (\cite{Ng} Proposition 4.2) that his property together with the corona factorization property is equivalent to having R\o rdam's `regularity' property. Hence any example of a C*-algebra with the corona factorization property, but not R\o rdam's property (if such an example exists), must fail to have Ng's property. Can there exist such a C*-algebra? To find one, by definition, we are looking for a non-unital (hereditary sub-)algebra with no non-zero bounded trace, and such that no matrix algebra over it is stable.\\

\noindent So what kind of criteria are there in general for the non-existence of (non-zero) bounded traces on a non-unital separable, exact, simple C*-algebra? 

Clearly, stability is a sufficient condition (\cite{HR}), but it is known not to be necessary. In \cite{R2} R\o rdam constructs a simple, separable, nuclear C*-algebra such that some matrix algebra over it is stable, but the algebra itself is not stable. This C*-algebra cannot have a bounded trace either, since such a trace would extend to a bounded trace on any matrix algebra. Is stability of sufficiently large matrix algebras the only obstruction to the existence of a (non-zero) bounded trace?

Let us consider the question of existence of bounded traces in the unital case. A famous result of Blackadar and Handelman (\cite{BH} II.4.11, cf.\ \cite{H} Theorem 2.4) says that a (non-zero) unital stably finite C*-algebra has a non-zero bounded quasitrace. Here, stable finiteness of the algebra is defined as stable finiteness of the unit as a projection in its stabilization (every multiple of the unit is finite). If we restrict our attention to exact C*-algebras, then by Haagerup's result \cite{Hp} quasitraces are traces and we can strengthen the Blackadar-Handelman theorem to having a non-zero bounded trace as its conclusion.

Is there a generalization of this result to the non-unital case? How would one have to define stable finiteness for a non-unital C*-algebra $C$? A rather weak way to do this is to ask for the unitization of $C$ to be stably finite. But this attempt is hopeless. The compact operators $\KK$ on a separable Hilbert space are stably finite in this definition, but there is no non-zero bounded trace on $\KK$ (\cite{HR}). A rather strong way is to call a non-unital C*-algebra stably finite, whenever its multiplier algebra has this property. Alternatively, instead of stable finiteness of the multiplier algebra, one could at least require that no matrix algebra over the multiplier algebra be properly infinite.\\

\noindent The question raised above, whether there exists a C*-algebra with the corona factorization property and not having R\o rdam's `regularity' property, is connected to the question of generalizing the Blackadar-Handelman theorem to the non-unital case. The non-existence of a non-unital C*-algebra as stipulated above (i.e., with no non-zero bounded trace and with no matrix algebra over it stable) would imply that, whenever a C*-algebra has no non-zero bounded trace, some matrix algebra over it has to be stable. But the multiplier algebra of a stable C*-algebra is properly infinite (\cite{R7}). So we have arrived at a proof of  both possible non-unital Blackadar-Handelman theorems. On the other hand, a counterexample to any of the non-unital Blackadar-Handelman theorems might provide us with a candidate for showing that R\o rdam's regularity property and the corona factorization property cannot coincide for exact C*-algebras.

It turns out that, while for suitably well behaved C*-algebras (for example in the sense of simple $\mathcal{Z}$-stable C*-algebras) the non-unital Blackadar-Handelman theorems hold, there are counterexamples in the general case. This paper contains the construction of two counterexamples. The first counterexample, given in Theorem \ref{nonsimpleexample}, is a stably commutative C*-algebra with no non-zero bounded trace, no non-zero unital quotient, and such that all matrix algebras over it have a finite multiplier algebra. The second counterexample, given in Theorem \ref{simpleexample}, is simple with no non-zero bounded trace, and such that no matrix algebra over it has a properly infinite multiplier algebra. (Since both algebras are exact there are also no quasitraces by Haagerup"s result.) In particular, both algebras constructed are non-unital C*-algebras without non-zero bounded trace, and such that no matrix algebra over them is stable; -- they are stably non-stable with no non-zero bounded trace. 

But the C*-algebras constructed do not have the corona factorization property either.  So, despite being counterexamples to any extension of  the Blackadar-Handelman theorem to the non-unital case, they fail to settle the question, raised by Ng in \cite{Ng}, of comparing the corona factorization property with R\o rdam's property. \\

\section{Preliminaries}\label{preliminaries}

We shall denote Murray-von Neumann equivalence of projections $p,q$ in a C*-algebra $A$ as usual by $p\sim q$, i.e., we write $p\sim q $ if there exists some partial isometry $v\in A$ such that $p=v^*v$ and $q=vv^*$. We shall write $p\preceq q$, and say that $p$ is subequivalent to the projection $q$ if $p$ is equivalent to a subprojection of $q$, i.e., if there exists a projection $p_0\in A$ such that $p\sim p_0\leq q$.

Finiteness and infiniteness of projections are defined in terms of subequivalence:

\begin{definition}\label{definfinite}
A projection $p$ in a C*-algebra $A$ is called finite, if $$p\sim q\leq p\mbox{ implies }q=p.$$
$p$ is called infinite if it is not finite. 
\end{definition}

\begin{definition}\label{stablyfinite}
A unital C*-algebra is called finite, if its unit is a finite projection. The C*-algebra is called stably finite if any multiple of the unit is a finite projection in $A\otimes\KK$.
\end{definition}

We will now recall the definition of traces and quasitraces.

\begin{definition}\label{deftrace}
A trace on a C*-algebra $A$ is a positive linear functional $\tau:A\rightarrow \CC$ such that $\tau(ab)=\tau(ba)$ for all $a,b\in A$. A trace on a unital C*-algebra is called a state if $\tau(1_A)=1$.
\end{definition}

\begin{definition}\label{defquasitrace}\cite{BH}
 A quasitrace is a continuous function $\tau:A^+\rightarrow\RR^+$ such that 
\begin{itemize}
\item $\tau(\lambda a)=\lambda \tau(a)$ for all $a\in A^+$ and all $\lambda\in\RR$,
\item $\tau(a+b)=\tau(a)+\tau(b)$ for all commuting elements $a,b\in A^+$,
\item $\tau(x^*x)=\tau(xx^*)$ for all $x\in A$, and such that 
\item $\tau$ extends to a map $\tau:M_2(A)^+\rightarrow \RR^+$ with the same properties. 
\end{itemize}
\end{definition}

Haagerup proved in \cite{Hp} that for unital exact C*-algebras quasitraces extend to traces. His result was extended to the non-unital case by Kirchberg in \cite{Ki}. In the following we will only work with exact C*-algebras and all results will consider traces.

We will say that a trace or quasitrace is bounded, if it is bounded in norm. Traces, defined as above, are automatically continuous and bounded (\cite{D} 2.1.8).\\

\noindent One of the main ingredients of our study is the multiplier algebra. For any C*-algebra $A$ there is (see e.g. \cite{L}) a unital C*-algebra $\mathcal{M}(A)$, unique up to isomorphism, such that
\begin{itemize}
\item[(a)] $\mathcal{M}(A)$ contains $A$ as an ideal,
\item[(b)] every ideal of $\mathcal{M}(A)$ has non-zero intersection with $A$, and
\item [(c)] $\mathcal{M}(A)$ is maximal in the sense that every unital C*-algebra satisfying  $(a)$ and $(b)$ embeds into $\mathcal{M}(A)$.
\end{itemize}
$\mathcal{M}(A)$ is called the multiplier algebra of $A$. If $A$ is itself unital, then $\mathcal{M}(A)=A$. For further details on multiplier algebras consult for example \cite{L}.\\

\noindent The corona factorization property, defined by Kucerovsky and Ng in \cite{KN2}, is a regularity property for C*-algebras.  It originated from the study of absorbing extensions \cite{EK}. There have been several papers on the corona factorization property and related concepts since (see for example \cite{R2}, \cite{R3}, \cite{KN}, \cite{KN3}, \cite{Ng}, \cite{OPR1}, \cite{OPR2}, \cite{BRTTW}, \cite{HRW}). It turns out that the corona factorization property can be defined in several equivalent ways, the most common one is to require for the stabilization of the given algebra that every full multiplier projection is properly infinite. (Recall that an element $a$ of a C*-algebra $A$ is called full if the closed two-sided ideal generated by $a$ is all of $A$.) It was shown by Kucerovsky and Ng in \cite{KN2} that this is equivalent to the following statement that we use as the definition.

\begin{definition}
Let $A$ be a separable C*-algebra. Then $A$ is said to have the corona factorization property, briefly the CFP, if whenever $D$ is a full hereditary subalgebra of $A\otimes \KK$ such that some (non-zero) matrix algebra $M_n(D)$ is stable, then $D$ is stable.
\end{definition}

There is a second regularity property that usually goes along with the study of the CFP. This property, introduced by R\o rdam in  \cite{R4}, is usually simply called  `regularity', but one might call it the tracial version of the corona factorization property.

\begin{definition}\label{defregular}
A C*-algebra is regular -- we shall also say that it has the tracial corona factorization property -- if every full hereditary subalgebra $D$ of $A\otimes \KK$ with no non-zero unital quotient and no non-zero bounded trace is stable.
\end{definition}

It is fairly easy to see that the tracial CFP implies the CFP (see e.g. \cite{Ng} Lemma 4.5)

\begin{proposition}\label{regularimpliescfp}
Let $A$ be a C*-algebra with the tracial CFP. Then $A$ has the CFP.
\end{proposition}

It is an open question whether the converse is true, i.e., whether the CFP and the tracial CFP are in fact equivalent properties for any C*-algebra. \\

\noindent In his survey (\cite{Ng}) on the corona factorization property, Ng introduces a new property, that he calls asymptotic regularity in the style of R\o rdam's property `regularity'. To continue with our revised terminology, let us call it the asymptotic tracial CFP. 

\begin{definition}\label{defasyregular}
A C*-algebra $A$ has the asymptotic tracial corona factorization property if, whenever $D$ is a full hereditary subalgebra of $A\otimes \KK$ with no non-zero unital quotient and no non-zero bounded race, there is some natural number $n\geq 1$ such that $M_n(D)$ is stable.
\end{definition}

Ng then proves the following result:

\begin{proposition}\label{regularequalscfppulsasycfp}\cite{Ng}
A separable C*-algebra has the tracial CFP if, and only if, it has both the CFP and the asymptotic tracial CFP.
\end{proposition}

It follows that if the tracial CFP is strictly stronger than the CFP, and if $A$ is an example of a C*-algebra with the CFP but without the tracial CFP, then $A$ cannot have the asymptotic tracial CFP.  For such an example, by definition, we are looking for a (hereditary sub-)C*-algebra $D$, with no non-zero bounded trace, no non-zero unital quotient, and such that no matrix algebra over $D$ is stable. Such a C*-algebra is given in Corollary \ref{stablynonstableexample}. But this example does not satisfy the corona factorization property either (\cite{KN3} Proposition 5.3), so it is of no help in deciding whether the CFP and the tracial CFP are equivalent.\\

\noindent In the classification program one of the central regularity properties is tensorial absorption of the Jiang-Su algebra, denoted by $\mathcal{Z}$, that was constructed in \cite{JS}. A C*-algebra $A$ that absorbs the Jiang-Su algebra tensorially, i.e., such that $A\otimes\mathcal{Z}\cong A$, is called $\mathcal{Z}$-stable. 

It can be seen from Theorem 3.6 of \cite{HRW} together with the fact that the Cuntz semigroup of any $\mathcal{Z}$-stable C*-algebra is almost unperforated (\cite{R5}) that every $\mathcal{Z}$-stable C*-algebra has the tracial corona factorization property, and hence the corona factorization property. (The latter can also be seen directly from Corollary 3.5 of \cite{HRW}.) 

The converse is not true: Kucerovsky and Ng show in \cite{KN1} that certain Villadsen algebras of the second type (\cite{V2}) do have the tracial corona factorization property, while perforation in their $K_0$-groups shows that these algebras cannot be $\mathcal{Z}$-stable.

\section{Motivation}

We are interested in a generalization of the following well-known theorem by Blackadar and Handelman (\cite{BH} II.4.11, cf. \cite{H} Theorem 4.2):

%
%

\begin{theorem}\label{blackadarhandelman}
Every unital stably finite C*-algebra admits a quasitrace.
\end{theorem}

The following related result is Theorem 2.5 of \cite{R7}. As the author, M. R\o rdam points out, the theorem is essentially due to Goodearl and Handelman (\cite{GH}) (and Haagerup's contribution is the passage from quasitraces to traces (\cite{Hp})).

\begin{theorem}\label{goodearlhandelmanhaagerup}
The following conditions are equivalent for an exact \underline{unital} C*-algebra:
\begin{itemize}
\item[(i)] $A$ has a non-zero, stably finite quotient.
\item[(ii)] $M_n(A)$ is not properly infinite for any (non-zero) $n\in \NN$.
\item[(iii)] $A$ admits a tracial state.
\end{itemize} 
\end{theorem}

Trivially, we get from this the equivalence of the following two statements  for a unital exact C*-algebra:
\begin{itemize}
\item $M_n(\mathcal{M}(A))$ is not properly infinite for any $n\in \NN$.
\item $A$ admits a non-zero bounded trace.
\end{itemize}
This leads to the question of a possible generalization to the non-unital case. Firstly, because in the unital setting existence of a trace is by the previous result connected to non-proper infinitness, one could hope for the following to be true:

The following statements are equivalent for an exact C*-algebra:
\begin{itemize}
\item $M_n(\mathcal{M}(A))$ is not properly infinite for any $n\in \NN$.
\item $A$ admits a non-zero bounded trace.
\end{itemize}

Secondly, after defining stable finiteness for non-unital C*-algebras as follows,
\begin{definition}\label{defstablyfinitefornonunital}
A non-unital C*-algebra is called stably finite, if its multiplier algebra is a stably finite algebra, i.e., the multiplier unit is a stably finite projection,
\end{definition}
\hspace{-15pt}one could hope for the following non-unital version of the Blackadar-Handelman theorem to be true:

\begin{center}Every stably finite exact C*-algebra admits a bounded trace.\end{center}

Theorem \ref{nonsimpleexample} and Theorem \ref{simpleexample} show that both of these non-unital statements (true in the unital case) are false in general. (Of course, it is enough to show that the second is false.) We will construct a simple C*-algebra, such that no matrix algebra over its multiplier algebra is properly infinite, but which has no non-zero bounded trace (Theorem \ref{simpleexample}).

 We will also construct a C*-algebra stably isomorphic to a commutative C*-algebra with the same properties (Theorem \ref{nonsimpleexample}). Theorem \ref{nonsimpleexample} shows that even the non-unital version of the original Blackadar-Handelman Theorem does not hold:

The non-simple C*-algebra without bounded trace constructed in this theorem is stably finite (in the sense of Definition \ref{defstablyfinitefornonunital}).\\

\noindent The counterexamples constructed are exact, so the algebras do not admit bounded quasitraces either by \cite{Hp}, despite being stably finite in a strong sense.

\section{Well-behaved C*-algebras}

For `well-behaved' C*-algebras (interpreted differently in the hypothesis of the following theorem and in one of its corollaries) we get positive results.

\begin{proposition}\label{asycfpisgood}
Every stably finite, exact C*-algebra $A$ with the asymptotic tracial corona factorization property and with the property that no matrix algebra over any non-zero unital quotient of $A$ is properly infinite, admits a non-zero bounded trace.
\end{proposition}

\begin{proof}
Suppose that $A$ has the properties of the hypothesis and assume $A$ has no non-zero bounded trace. Then $A$ has no unital quotient either, because Theorem \ref{goodearlhandelmanhaagerup} applied to this unital C*-algebra would give a tracial state on a quotient, and so a non-zero bounded trace on $A$. Now the asymptotic tracial CFP implies stability of $M_n(A)$ for some $n\in\NN$. But then $\mathcal{M}(M_n(A))$ is properly infinite (\cite{R7}, Lemma 3.4), which contradicts the hypothesis. 
\end{proof}

\begin{remark}
Note that there exist finite non-unital C*-algebras with unital purely infinite quotients (Example 2.4 of \cite{R7}). But it is possible that such algebras can not be stably finitie in the sense of Definition \ref{defstablyfinitefornonunital} and that the last assumption of the proposition is redundant.
\end{remark}

\begin{corollary}
Every simple, stably finite, exact C*-algebra with the asymptotic tracial corona factorization property admits a non-zero bounded trace.
\end{corollary}

Lemma 4.5 from \cite{Ng} gives: 

\begin{corollary}\label{regularisgood}
Every simple, stably finite, exact C*-algebra with the tracial corona factorization property admits a non-zero bounded trace.
\end{corollary}

The last paragraph of Section \ref{preliminaries} finally gives us the following result.

\begin{corollary}\label{zstableisgood}
Every simple, stably finite, exact $\mathcal{Z}$-stable C*-algebra admits a non-zero bounded trace.
\end{corollary}

\section{A non-simple example}

But there do exist counterexamples in the general case. The construction is an application of the results from \cite{Pe}. The following theorem is a combination of Corollary 4.5 (and its proof) and Corollary 5.4 of that paper.

\begin{theorem}\label{nonfullprojectionexample}
There exists a connected compact Hausdorff space $X$ and a multiplier projection $$Q=\bigoplus_{j=1}^\infty p_j\in\mathcal{M}(C(X,\KK)),$$
where each $p_j$ is a projections in $C(X,\KK)$ of rank one, such that $Q$ is stably finite.
\end{theorem}

We will start with the construction for the non-simple example. 

\begin{theorem}\label{nonsimpleexample}
There exists a separable (non-simple), exact (non-unital), stably finite C*-algebra with no non-zero bounded trace. 
\end{theorem}

\begin{proof}
Let $X$ be a Hausdorff space and let $Q$ a multiplier projection be as in the statement of Theorem \ref{nonfullprojectionexample}. Let $C:=C(X,\KK)$. Define $B$ to be the cut-down C*-algebra 
$$B:=QCQ.$$
Then $B$ is stably finite (since $\mathbbm{1}_{M_n(B)}=n\cdot Q$). Also, $B$ is exact.\\

Assume there is a non-zero bounded trace $\tau$ on $B$. Consider the sequence $(Q_n)_{n\in \NN}$ of subprojections of $Q$ given by
$$Q_n:=\bigoplus_{j=1}^n p_j,$$
and consider the corresponding cut-down C*-algebras
$$D_n:=Q_n C Q_n.$$
Then each $D_n$ is a unital hereditary subalgebra of $D$ and the sequence $(D_n)_{n\in \NN}$ is increasing in the sense that each $D_n$ is a hereditary subalgebra of $D_{n+1}$ for every $n$ in $\NN$. The bounded trace $\tau$  gives for each $n\in \NN$, by restriction, a bounded trace $\tau_n$ on $D_n$, such that for each projection $p\in D_n$ we have 
$$\tau(p)=\tau_m(p)=\tau_n(p)$$
whenever $m\geq n$.

The bounded trace $\tau_n$ is a scalar multiple of a tracial state on $D_n$. Using Swan's Theorem, $D_n$ is a hereditary subalgebra of some matrix algebra over $C(X)$. By the Riesz representation theorem together with the fact that states extend uniquely from hereditary subalgebras (\cite{P} Proposition 3.1.6), one sees that any tracial state $\sigma$ on $D_n$ is of the form
$$\sigma(a)=\frac{1}{rank(Q_n)}\int Tr(a(x))d\mu(x),$$
where $\mu$ is a Borel probability measure on the compact Hausdorff space $X$, and $Tr$ denotes the standard trace on matrix algebras over the complex numbers. The tracial state space is a simplex (\cite{BH} II.4.4), and its extreme points correspond to the Dirac measures $\mu_x$ (see e.g. \cite{S} Example 8.16), i.e., the extreme points are the tracial states given by
$$\sigma_{x_0}(a)=Tr(a(x_0))$$
for some $x_0\in X$. It follows that for any projection $p\in D_n$ and for any extreme point $\sigma$ of the tracial state space we have $$\sigma(p)=\frac{rank(p)}{rank(Q_n)},$$ hence the same must be true for any tracial state (being, by the Krein-Milman theorem, in the closed convex hull of the extreme points).

Therefore there is some $k\in\RR$ such that $\tau_n(p)=k\cdot rank(p)$ for every $n\in\NN$ and $p\in D_n$. The number $k$ is strictly positive. Indeed, $k\geq 0$ because traces are positive. Also number $k$ cannot be zero because, by continuity of $\tau_n$, the closure of the set \mbox{$\{c\in C\ |\ \tau(c)=0\}$} is an ideal of $C$ and every projection in $D_n$ is full.

Hence, the sequence of projections $Q_n$ in $C$ satisfy
$$\tau(Q_n)=\tau_n(Q_n)=n \cdot k\stackrel{n\rightarrow \infty}{\longrightarrow}\infty.$$
This contradicts the assumption that $\tau$ is bounded.


\end{proof}

\section{A simple example}

We now turn to the simple case, in which we only prove that the matrix algebras over the multiplier algebra of the constructed C*-algebra $B$ are not properly infinite, but stable finiteness might hold as well. Let us set the following notation, which is adapted from \cite{R1}. Let 
$$p\in \mbox{C}(S^2,M_2(\CC)) \mbox{ denote the Bott projection,}$$
i.e., the projection corresponding to the `Hopf bundle' $\xi$ over $S^2$ with total Chern class $c(\xi)=1+x$.

With $n,N\in \NN$ and $1\leq n\leq N$, let $\pi_n:\prod_{j=1}^N S^2 \rightarrow S^2$ denote the coordinate projection onto the $n$-th coordinate. Consider the (orthogonal) projection
$$p_n:=p\circ \pi_n \in \mbox{C}(\prod_{j=1}^N S^2,M_2(\CC)).$$
If $I\subseteq \{1,2,\ldots,N\}$ is a finite subset, $I=\{n_1,n_2,\ldots n_k\}$, then let $p_I$ denote the pointwise tensor product 
$$p_I:=p_{n_1}\otimes p_{n_2} \otimes \ldots \otimes p_{n_k}\ \in \mbox{C} (\prod_{j=1}^N S^2,M_2(\CC)\otimes M_2(\CC)\otimes\ldots \otimes M_2(\CC)).$$

Using the well-known correspondence between complex vector bundles and Murray-von Neumann equivalence classes of projections, it is shown in \cite{R1} that the projection $p_n$ corresponds to the pull-back of the Hopf bundle via the coordinate projection $\pi_n$, denoted by $\xi_n:=\pi_n^*(\xi)$, and that the projection $p_I$ corresponds to the tensor product of vector bundles $\xi_{n_1}\otimes \xi_{n_2}\otimes \ldots \otimes \xi_{n_k}$. We will denote any projection that corresponds to the trivial (complex) line bundle by $e$.\\

Considering the compact operators $\KK$ on a separable Hilbert space as an AF algebra, the inductive limit of the sequence
$$\CC\rightarrow M_2(\CC)\rightarrow M_3(\CC)\rightarrow M_4(\CC)\rightarrow \ldots ,$$
with connecting $^*$-homomorphisms mapping each matrix algebra into the upper left corner of any larger matrix algebra at a later stage, we get an embedding of each matrix algebra over $\CC$ into the compact operators $\KK$. In this way we can consider all the projections $p_n$ and $p_I$, defined as above, as projections in C$(\prod_{j=1}^N S^2,\KK)$\\

\noindent We will need the following technical result, which is an improvement to results obtained by R\o rdam in \cite{R1}.

\begin{proposition}\label{technicalresult}[cf. \cite{Pe}, Proposition 3.2]
Let $n,N\in\NN$. For each $ 1\leq j\leq n$ let $I_j$ be a finite subset of $\{1,2,\ldots,N\}$ and consider the projection $Q$ in $C(\prod_{j=1}^N S^2,\KK))$ given by
$$Q=\bigoplus_{j=1}^n p_{I_j}.$$
Let $m\in\NN$. \\
Then the following statements are equivalent:
\begin{itemize}
\item[(i)] $m\cdot e\npreceq Q= \bigoplus_{j=1}^{n}p_{I_j}$
\item[(ii)] $|F| < \left | \bigcup_{j\in F} I_j \right |+m$ for all finite subsets $F\subseteq\{1,2,\ldots,n\}$. 
\end{itemize}
\end{proposition}

\noindent The criterion for checking that a projection is not properly infinite is given by the following Lemma (\cite{Pe}, Lemma 5.1):
\begin{lemma}\label{notpicriteria}
Let $A$ be a C*-algebra and $p$ and $q$ two projections in $A \otimes \KK$ such that $p\preceq k\cdot q$, but $p \npreceq m\cdot q$ for some $m<k$. Then $q$ is not properly infinite.
\end{lemma}

\noindent We are now ready for the construction of the simple example contradicting a possible generalization of the Blackadar-Handelman theorem.

\begin{theorem}\label{simpleexample}
There exists a separable, simple, exact (non-unital) C*-algebra $B$ such that $B$ has no non-zero bounded trace and such that no (non-zero) matrix algebra over the multiplier algebra of $B$ is properly infinite.
\end{theorem}

\begin{proof}
The construction is closely related to the one M. R\o rdam used to construct an example of a simple C*-algebra $A$, such that $M_n(B)$ is stable, but $M_k(B)$ is not stable for any $k<n$ (\cite{R2}). The idea of his construction furthermore goes back to ideas of Villadsen, who was the first one to construct a C*-algebra with perforation in its ordered $K_0$-group (\cite{V1}).

Let $X_1=S^2$ be the two-sphere. Then choose for each $j\in\NN$, $j\geq 2$, a compact Hausdorff space $X_j$ as a product of two-spheres. The number of copies of the two-sphere is  recursively defined by
$$X_{j+1}=X_j^{k_j}\times (S^2)^{m_{j+1}(j+1)}$$
where the $k_j$ are natural numbers such that
$$\sum_{s=1}^\infty\left [\left ( 1-\prod_{r=s}^{\infty}\frac{k_r}{k_r+1}  \right )\right ]= R<\infty,$$
and the $m_j,\ j\in\NN$, are recursively defined as $m_1=1$, $m_{j+1}=(k_j+1)m_j$. This gives
$$m_j=\prod_{i=1}^{j-1}(k_i+1)$$
and both the recursive and the explicit formula will be useful in the sequel.\\

We will define our C*-algebra $B$ as a hereditary subalgebra of the stabilization of a certain AH-algebra. So our first step should be the construction of the AH-algebra. For this define for each $j$ a homogeneous algebra $A_j$ by
$$A_j:=C(X_j,\KK)\cong C(X_j)\otimes \KK.$$
With $\pi_l^j:X_{j+1}\rightarrow X_j$ the projection map onto the $l$-th copy of $X_j$ in $X_ {j+1}$ ($l=1,2,\ldots,k_j$), and $c_j\in X_j$ a point to be specified, the connecting maps are given by 
$$\varphi_j:A_j\rightarrow M_{(k_{j}+1)}(A_j)\cong  A_{j+1},  $$
$$\varphi_j(f)(x)=\mbox{diag}(f\circ\pi_1^j(x),f\circ\pi_2^j(x),\ldots,f\circ\pi_{k_j}^j(x),f(c_j)).$$
Note that the multiplicity of $\varphi_j$ is $(k_{j}+1)$. To consider compositions of connecting maps we define
$$\varphi_{i,j}:=\varphi_{i-1}\circ \varphi_{i-2}\circ \ldots\circ \varphi_{j+1}\circ \varphi_j :A_j\rightarrow A_i\  (i> j),$$
$$k_{i,j}:=\prod_{n=j}^{i-1}k_n \mbox{ (the number of projection maps in $\varphi_{i,j}$)},\mbox{ and} $$
$$l_{i,j}:=\prod_{n=j}^{i-1}(k_n+1)-\prod_{n=j}^{i-1}k_n  \mbox{ (the number of point evaluations in $\varphi_{i,j}$)}.$$
Note that $k_{i,j}+l_{i,j}=\prod_{n=j}^{i-1}(k_n+1)$=multiplicity of $\varphi_{i,j}$. Then 
$$\varphi_{i,j}(f)\sim \mbox{diag}(f\circ \pi_1^{i,j},f\circ \pi_2^{i,j},\ldots,f\circ \pi_{k_{i,j}}^{i,j},f(c_1^{i,j}),f(c_2^{i,j}),\ldots f(c_{l_{i,j}}^{i,j})),$$
with coordinate projection maps $\pi_1^{i,j},\pi_2^{i,j},\ldots,\pi_{k_{i,j}}^{i,j}:X_i\rightarrow X_j$ and points $c_k^{i,j}\in X_j$.

 Define $X_j^i:=\{c_1^{i,j},c_2^{i,j},\ldots,c_{l_{i,j}}^{i,j}\}$, i.e., $X_j^i$ is the set of points in $X_j$ that appear as point evaluations in $\varphi_{i,j}$. 
Then 
$$X_{j}^i=X_j^{i-1}\cup \{\pi_1^{i,j}(c_i),\pi_2^{i,j}(c_i),\ldots,\pi_{k_{i,j}}^{i,j}(c_i)\}\ , \ i>j+1,$$
and $X_j^{j+1}=\{c_j\}$.
Choose the points $c_j$ in such a way that the set $$\bigcup_{r=j+1}^{\infty} X_j^r
$$ is dense in $X_j$ for all $j\in\NN$.

By Proposition 2.1 of \cite{DNP} the choice of point evaluations implies that the C*-algebra $A$, given as the inductive limit 
$$\xymatrix{ A_1\ar[r]^{\varphi_1}  & A_2\ar[r]^{\varphi_2}\ar@/_1pc/[rr]_{\varphi_{j,2}} & \ldots\ar[r]^{\varphi_{j-1}} & A_j\ar[r]^{\varphi_j} \ar@/_1pc/[rr]_{\varphi_{\infty,j}} & \ldots\ar[r] & A},$$
is simple. Moreover, $A$ is exact and separable by construction. \\

As mentioned before, the C*-algebra $B$ we are looking for is a hereditary subalgebra of $A\otimes \KK$ defined as the cut-down algebra $B=Q(A\otimes \KK)Q$ coming from a certain multiplier projection $Q$. So our next step is to describe the multiplier projection $Q\in\mathcal{M}(A\otimes \KK)$:

Let $\varphi_{\infty,j}:A_j\rightarrow A$ denote the canonical maps from the building block algebras $A_j$ into the inductive limit algebra $A$.

Denote by $q_1\in A_1=C(S^2,\KK)$ the Bott projection, and set $Q_1:=\varphi_{\infty,1}(q_1)$.

Denote by $q_2\in A_2=C((S^2)^{k_1+2m_2},\KK)$ the projection given by the direct sum of $m_2$ tensor products of two Bott projections such that each of the last $2m_2$ coordinates of $X_2=X_1^{k_1}\times (S^2)^{2m_2}$ is used precisely once, i.e.,
$$q_2:=\bigoplus_{\alpha=1}^{m_2}p_{I_\alpha^2}$$
$$\mbox{where }|I_\alpha^2|=2,\ I_\alpha^2\cap I_\beta^2=\emptyset\ \mbox{ for $\alpha\neq \beta$, and }\bigcup_{\alpha=1}^{m_2}I_\alpha=\left ( \mbox{last $2m_2$ coordinates of $X_2$}\right ).$$ 
Set $Q_2:=\varphi_{\infty,2}(q_2)$.

Similarly for $j\geq 3$, continue as follows.

Denote by $q_j\in A_j=C(X_{j-1}^{k_{j-1}}\times(S^2)^{jm_j},\KK)$ the projection given by the direct sum of $m_j$ tensor products of $j$ Bott projections such that each of the last $jm_j$ coordinates of $X_j=X_{j-1}\times (S^2)^{jm_j}$ is used precisely once, i.e.,
$$q_j:=\bigoplus_{\alpha=1}^{m_j}p_{I_\alpha^j}$$
$$\mbox{where }|I_\alpha^j|=j,\ I_\alpha^j\cap I_\beta^j=\emptyset\ \mbox{ for $\alpha\neq \beta$, and}\bigcup_{\alpha=1}^{m_j}I_\alpha=\left ( \mbox{last $jm_j$ coordinates of $X_j$}\right ).$$ 
Set $Q_j:=\varphi_{\infty,j}(q_j)$.

The multiplier projection in $\mathcal{M}(A\otimes \KK)$ we are looking for is given by the infinite direct sum
$$Q:=\bigoplus_{j=1}^\infty Q_j.$$
With this definition of $Q$, the algebra $B$ in the statement of the theorem is given, as we shall now see, by
$$B:=Q(A\otimes \KK)Q.$$
$B$ is separable, simple and exact, because these properties pass to hereditary subalgebras. Also, $B$ is non-unital, with $\left \{Q_1\oplus Q_2\oplus\ldots \oplus Q_j\right \}_{j\in\NN}$ as an approximate unit of projections.\\

Let $e_j\in A_j$ be a trivial 1-dimensional projection, and let $E:=\varphi_{\infty,1}(e_1)$. Then, according to multiplicities of the connecting maps $\varphi_j$, we get that $\varphi_{j,1}(e_1)\sim m_j\cdot e_j$. (Recall that we write $m_j\cdot e_j$ for $e_j\otimes \mathbbm{1}_{m_j}$.)\\

We will now study the comparison of the projections $Q_j=\varphi_{\infty,j}(q_j)$ with the projection $E=\varphi_{\infty,1}(e_1)$ in $A$: By definition of the projections $q_j$ in $A_j$ and by Proposition \ref{technicalresult} we have
$$(j+1)\cdot q_j\succeq m_j\cdot e_j=\varphi_{j,1}(e_1)\mbox{, and}$$
$$j\cdot q_j\nsucceq e_j.$$
It follows from the first result that
$$E=\varphi_{\infty,j}({\varphi_{j,1}}(e_1))\preceq \varphi_{\infty,j}((j+1)\cdot q_j)=(j+1)\cdot Q_j.$$
Define further for $i\geq j$,$$f_{i,j}:=\varphi_{i,1}(q_1)\oplus\varphi_{i,2}(q_2)\oplus\ldots\oplus\varphi_{i,j}(q_j)\in A_i.$$ 
Then $f_{i,j}=\varphi_{i,j}(f_{j,j})$ and 
$$\varphi_{\infty,i}(f_{i,j})=Q_1\oplus Q_2\oplus\ldots \oplus Q_j.$$
Let us further investigate the  projections $f_{i,j}\in A_i$ and $f_{j,j}\in A_j$. To start with $f_{j,j}$, notice that
$$\varphi_{j,1}(q_1)\sim \bigoplus_{\alpha=1}^{k_{j,1}} p_{J_{j,\alpha}^1}\oplus \left (l_{j,1}\cdot e_j\right ),$$
where for any fixed $j$ the index sets $J_{j,\alpha}^1$ are pairwise disjoint and of cardinality $|J_{j,\alpha}^1|=1$. Further, 
$$\varphi_{j,2}(q_2)\sim \bigoplus_{\alpha=1}^{m_2k_{j,2}} p_{J_{j,\alpha}^2}\oplus \left (m_2l_{j,2}\cdot e_j\right )$$
with $|J_{j,\alpha}^2|=2$, and for each fixed $j$ the index sets $J_{j,\alpha}^1,J_{j,\beta}^2$ are pairwise disjoint.
Accordingly, for $3\leq s\leq j$:
$$\varphi_{j,s}(q_s)\sim \bigoplus_{\alpha=1}^{m_sk_{j,s}} p_{J_{j,\alpha}^s}\oplus \left (m_sl_{j,s}\cdot e_j\right ),$$
with $|J_{j,\alpha}^s|=s$, and for fixed $j$ the index sets $J_{j,{\alpha_1}}^1,J_{j,{\alpha_2}}^2,\ldots,J_{j,{\alpha_s}}^s$ are pairwise disjoint.

It follows that 
$$f_{j,j}\sim\left ( \bigoplus_{s=1}^j\left ( \bigoplus_{\alpha=1}^{m_sk_{j,s}} p_{J_{j,\alpha}^s}\right )\right ) \oplus \left ( \bigoplus_{s=1}^j \left (m_sl_{j,s}\cdot e_j\right )\right ).$$
Since the numbers $k_{i,j}$ tell us the number of direct summands coming from coordinate projections of $\varphi_{i,j}$ and the $l_{i,j}$ tell us about the number of point evaluations of $\varphi_{i,j}$ we get more generally for any $i\geq j:$
$$f_{i,j}\sim \left (\bigoplus_{s=1}^j\left ( \bigoplus_{\alpha=1}^{m_sk_{i,s}} p_{J_{i,\alpha}^s}\right )\right ) \oplus \left ( \bigoplus_{s=1}^j \left (m_sl_{i,s}\cdot e_i\right )\right ).$$
In other words, for each $i\geq j$, the projection $f_{i,j}$ consists of a direct sum of $\left( \sum_{s=1}^jm_sl_{i,s}\right )$ trivial projections and a projection $\bigoplus_{s=1}^j r_s^i$ in $A_i$, where each summand $r_s^i$ is itself a direct sum of $(m_sk_{i,s})$ tensor products, each of $s$ Bott projections, such that all Bott projections involved come from distinct coordinates of $X_i$. In particular $e_i$ is not a subprojection of $\bigoplus_{s=1}^j r_s^i$ by Proposition \ref{technicalresult}.\\

To prove that no matrix algebra over the multiplier algebra of $B$ is properly infinite we will need to consider trivial subprojections of multiples of $f_{i,j}$ as well. Let $n\in\NN$. By Proposition \ref{technicalresult},
$$n\cdot f_{i,j}\succeq \left ( \left (\sum_{s=1}^jm_sk_{i,s} \left ( \max\{0,(n-s)\} \right ) \right )+\left (n \sum_{s=1}^j  m_sl_{i,s} \right )\right )\cdot e_i=:a_{i,j}\cdot e_i,$$
and by the same proposition this inequality is sharp -- that is, $n\cdot f_{i,j}\nsucceq (a_{i,j}+1)\cdot e_i$.\\

Let us show that for all $j\in \NN$:
 $$n\cdot (Q_1\oplus Q_2\oplus \ldots \oplus Q_j)\nsucceq M(n,R)\cdot E,$$
where $M(n,R)$ is the function of $n$ and $R$ given by
$$M(n,R)=\sum_{s=1}^{n}(n-s)+nR=\frac{n(n-1)}{2}+nR.$$
For $i\geq j$ we compute:
$$\frac{a_{i,j}}{m_i}=  \left ( \sum_{s=1}^j\frac{m_s}{m_i}k_{i,s} \left ( \max\{0,(n-s)\} \right ) \right )+\left (n\sum_{s=1}^j \frac{m_s}{m_i}l_{i,s}\right ) $$
$$=\sum_{s=1}^j\left [\frac{\prod_{r=1}^{s-1}(k_r+1)}{\prod_{r=1}^{i-1}(k_r+1)}\left (\prod_{r=s}^{i-1}k_r\right )\left ( \max\{0,n-s\}\right )\right ] \hspace{2cm}$$
$$ \hspace{2cm} +n\sum_{s=1}^j\left [\frac{\prod_{r=1}^{s-1}(k_r+1)}{\prod_{r=1}^{i-1}(k_r+1)}\left ( \prod_{r=s}^{i-1}(k_r+1)-\prod_{r=s}^{i-1}k_r  \right )\right ] $$
$$=\sum_{s=1}^j\left [\prod_{r=s}^{i-1}\frac{k_r}{k_r+1}\left ( \max\{0,n-s\}\right )\right ]+n\sum_{s=1}^j\left [\left ( 1-\prod_{r=s}^{i-1}\frac{k_r}{k_r+1}  \right )\right ]$$
$$<\sum_{s=1}^j\left[ \left ( \max\{0,n-s\}\right )\right ]+ n\sum_{s=1}^\infty\left [\left ( 1-\prod_{r=s}^{\infty}\frac{k_r}{k_r+1}  \right )\right ]$$
$$\leq \sum_{s=1}^{n}(n-s)+nR$$
$$=M(n,R).$$
It follows that
$$a_{i,j}<M(n,R)m_i\ \mbox{ for all $i$ and $j$ with $i\geq j$.}$$
This shows that 
$$n\cdot f_{i,j}\nsucceq M(n,R)\cdot \varphi_i(e_1)$$
in $A_i$ for all $i$ and $j$ with $i\geq j$. Hence,
$$n\cdot (Q_1\oplus Q_2\oplus \ldots \oplus Q_j)\nsucceq M(n,R)\cdot E$$
for all $j\in \NN$ by standard inductive limit arguments.\\

Next, let us verify that for $j\in\NN$ large enough and $n>R$,
$$2n\cdot (Q_1\oplus Q_2\oplus \ldots \oplus Q_j)\succeq M(n,R)\cdot E.$$

As above we can compute the number of trivial subprojections using Proposition \ref{technicalresult} and see that
$$2n\cdot f_{i,j}\succeq \left ( \left ( \sum_{s=1}^j m_sk_{i,s}\left ( \max\left \{0,2n-s\right \} \right )\right )+ \left (2n\sum_{s=1}^jm_sl_{i,s}\right ) \right ) \cdot e_i=:b_{i,j}\cdot e_i.$$
Consider the equation
$$\frac{b_{i,j}}{m_i}=\sum_{s=1}^j\left (\prod_{r=s}^{i-1}\frac{k_r}{k_r+1}\left ( \max\{0,2n-s\}\right )\right )+2n\sum_{s=1}^j\left (1-\prod_{r=s}^{i-1}\frac{k_r}{k_r+1}\right ).$$
If we can show that $\frac{b_{i,j}}{m_i}\geq M(n,R)$ for some $i,j\in\NN$, then 
$$2n\cdot f_{i,j}\succeq M(n,R)m_i\cdot e_i,\mbox{ and hence }2n\cdot (Q_1\oplus Q_2\oplus \ldots \oplus Q_j)\succeq M(n,R)\cdot E,$$
as desired. Therefore, we need only show that for large enough $i$ and $j$ the following inequality holds:
$$\sum_{s=1}^j\left (\prod_{r=s}^{i-1}\frac{k_r}{k_r+1}\left ( \max\{0,2n-s\}\right )\right )+2n\sum_{s=1}^j\left (1-\prod_{r=s}^{i-1}\frac{k_r}{k_r+1}\right )\geq M(n,R).\hspace{1cm}(*)$$

Choose $j\geq 2n$, and recall that we are only considering the case $n>R$. Then, for any $i\geq j$,

$$\sum_{s=1}^j\left (\prod_{r=s}^{i-1}\frac{k_r}{k_r+1}\left ( \max\{0,2n-s\}\right )\right )+2n\sum_{s=1}^j\left (1-\prod_{r=s}^{i-1}\frac{k_r}{k_r+1}\right )$$
$$\geq \sum_{s=1}^{2n-1}\left [\left (\prod_{r=s}^{i-1}\frac{k_r}{k_r+1}\left ( 2n-s\right )\right )+ 2n \left (1-\prod_{r=s}^{i-1}\frac{k_r}{k_r+1}\right )\right ]$$
$$= \sum_{s=1}^{2n-1}\left [ 2n -s\left (\prod_{r=s}^{i-1}\frac{k_r}{k_r+1}\right )\right ]$$
$$\geq \sum_{s=1}^{2n-1}\left ( 2n -s\right )=\frac{2n(2n-1)}{2}$$
$$=n(n+n-1)=\frac{n(n-1)}{2}+n\left ( \frac{n-1}{2} +n\right )$$
$$>\frac{n(n-1)}{2}+nR$$
$$= M(n,R).$$

This completes the proof of ($\ast$), and this again completes the proof of 
$$2n\cdot (Q_1\oplus Q_2\oplus \ldots \oplus Q_j)\succeq (M(n)+nR)\cdot E$$
whenever $n>R$ and $j\geq n$.\\


We can now prove that for no choice of $n\in\NN$ is the C*-algebra $$M_n(\mathcal{M}(B))\cong (n\cdot Q)\mathcal{M}(A\otimes\KK)(n\cdot Q)$$ properly infinite. In other words, the unit $n\cdot Q$ of this algebra is not properly infinite for any $n\in \NN$.

\hspace{0.4cm}It is enough to show that $n\cdot Q$ is not properly infinite for all $n>R$, as proper infiniteness of some multiple of $Q$ implies proper infiniteness of any higher multiple. 
 
For $n>R$ the projection $\left (2n\cdot Q\right )$ majorizes $M(n,R)$ copies of the projection $E$, but $\left (n\cdot Q\right )$ does not, so $\left (n\cdot Q\right )$ can not be properly infinite by Lemma \ref{notpicriteria}.\\


It only remains to show that $B$ has no non-zero bounded trace. So let $\tau$ be a non-zero trace on $B$. The trace $\tau$ must be faithful by simplicity of $B$. So $\tau(E)>0$.

The trace $\tau$ defines a trace on $A$, which we also denote by $\tau$. Then this trace defines a trace on the building block algebras $A_j$ by
$$\tau_j:=\tau\circ \varphi_{\infty , j}.$$
Since $\tau_j$ is a trace on the stabilization of a commutative algebra, and since $q_j$ and $\varphi_{j,1}(e_1)$ are of the same rank (equal to $m_j$), the trace $\tau_j$ agrees on these two projections. Hence
$$0<\tau(E)=\tau(\varphi_{\infty,j}(e_j))=\tau_j(\varphi_{j,1}(e_1))=\tau_j(q_j)=\tau(Q_j) \mbox{ for all $j$,}$$
and hence
$$\tau(Q_1\oplus Q_2\oplus\ldots \oplus Q_k)= k\cdot \tau(E).$$
Since for any $k\in \NN$ the projection $Q_1\oplus Q_2\oplus\ldots \oplus Q_k$ is an element of $B$, the trace $\tau$ is unbounded.

\end{proof}

\begin{corollary}\label{stablynonstableexample}
There exists a separable, simple, exact, non-unital C*-algebra $B$ such that $B$ has no non-zero bounded trace and such that no matrix algebra over $B$ is stable.
\end{corollary}

\begin{proof}
Consider the C*-algebra $B$ of the previous theorem. If $M_n(B)$ were stable for some $n\in\NN$, then the unit of $\mathcal{M}(M_n(B))$ would be properly infinite (\cite{R7} Lemma 3.4), in contradiction to the previous theorem.
\end{proof}

For simple separable stably finite C*-algebras with the tracial corona factorization property, non-existence of a bounded trace implies the C*-algebra to be stable. The examples in \cite{R2} show that another obstruction to existence of boubded traces is that large enough matrix algebras are stable (while the algebra might not be stable itself). The result of Corollary \ref{stablynonstableexample} says that this is not the end of the story and that there are further obstructions to the existence of bounded traces for general C*-algebras.

{
\small
\bibliographystyle{ieeetr}

\end{document}